\documentclass[11pt]{article}

\usepackage{mathrsfs}
\usepackage{nicefrac}
\usepackage[T1]{fontenc}
\usepackage{latexsym,amssymb,amsmath,amsfonts,amsthm}
\usepackage{graphicx}
\usepackage{caption}
\usepackage[utf8]{inputenc}
\usepackage{authblk}
\usepackage{bm}
\usepackage{accents}
\usepackage{booktabs}  
\usepackage[pdftex,colorlinks=true,citecolor=blue,linkcolor=blue]{hyperref}

\newtheorem{theorem}{Theorem}[section]
\newtheorem{lemma}{Lemma}[section]

\newtheorem{remark}{Remark}[section]

\makeatletter

\begin{document}

\title{A new estimate for a quantity involving the Chebyshev polynomials of the first kind}
\author{Xuefeng Xu%
\thanks{Corresponding author. LSEC, Institute of Computational Mathematics and Scientific/Engineering Computing, Academy of Mathematics and Systems Science, Chinese Academy of Sciences, Beijing 100190, China, and School of Mathematical Sciences, University of Chinese Academy of Sciences, Beijing 100049, China (\texttt{xuxuefeng@lsec.cc.ac.cn}).} 
\quad and \quad 
Chen-Song Zhang%
\thanks{LSEC \& NCMIS, Academy of Mathematics and Systems Science, Chinese Academy of Sciences, Beijing 100190, China (\texttt{zhangcs@lsec.cc.ac.cn).}}}
\date{\today}
\maketitle
	
\begin{abstract}

In this paper, we establish a new estimate (including lower and upper bounds) for an important quantity involved in the convergence analysis of smoothed aggregation algebraic multigrid methods. The new upper bound improves the existing ones. And our upper bound is optimal.

\end{abstract}
	
	
\smallskip
	

\section{Introduction}

The celebrated Chebyshev polynomials have a wide range of applications in many fields such as numerical analysis, differential equations, approximation theory, and number theory (see, e.g.,~\cite{Fox1968,Rivlin1990,Mason2003}). Some extensions and recent developments of the Chebyshev polynomials can be found, e.g., in~\cite{Witula2006,Moreno2012,Kroo2013,Totik2014,Tatarczak2016,Mesk2017,Schiefermayr2017}. As is well known, the Chebyshev polynomials of the first kind are defined by the recurrence relation
\begin{equation}\label{Cheby}
T_{0}(t)=1, \quad T_{1}(t)=t, \quad T_{k+1}(t)=2tT_{k}(t)-T_{k-1}(t) \quad k=1,2,\ldots
\end{equation}
In particular, if $t\in[-1,1]$, $T_{k}(t)$ can be explicitly expressed as
\begin{equation}\label{Cheby2}
T_{k}(t)=\cos(k\arccos t).
\end{equation}

Smoothed aggregation algebraic multigrid (SA-AMG) has become a popular method for solving large sparse linear systems arising from the discretizations of elliptic partial differential equations; see, e.g.,~\cite{Vanek1996,Vanek2001,Brezina2012,Vanek2012,Hu2016}. The convergence analysis of SA-AMG developed in~\cite{Brezina2012} involves an important quantity $\frac{C_{\nu}}{2\nu+1}$ ($\nu=1,2,\ldots$), where
\begin{equation}\label{C_v}
C_{\nu}:=\sup_{t\in(0,1]}\frac{\left|1-\frac{(-1)^{\nu}T_{2\nu+1}(\sqrt{t})}{(2\nu+1)\sqrt{t}}\right|}{\sqrt{t}}.
\end{equation}
It was proved by Brezina, Van\v{e}k, and Vassilevski~\cite[Proposition A2]{Brezina2012} that, for any positive integer $\nu$,
\begin{equation}\label{old1}
\frac{C_{\nu}}{2\nu+1}\leq 2, 
\end{equation}
which plays a key role in the convergence analysis of SA-AMG. Recently, Guo and Xiang~\cite[Proposition 2.2]{Guo2014} sharpened the estimate~\eqref{old1} and obtained that, for any positive integer $\nu$,
\begin{equation}\label{old2}
\frac{C_{\nu}}{2\nu+1}\leq 0.885.
\end{equation}

From the numerical results in Table~\ref{tab:1} below, we observe that the optimal upper bound for $\frac{C_{\nu}}{2\nu+1}$ is likely to be $\frac{C_{1}}{3}=\frac{4}{9}$. In this paper, we prove this conjecture and derive that
\begin{equation}\label{imp_est}
\frac{1}{4}\leq\frac{C_{\nu}}{2\nu+1}\leq\frac{4}{9}
\end{equation}
for any positive integer $\nu$. Obviously, the upper bound in~\eqref{imp_est} improves the existing ones in~\eqref{old1} and~\eqref{old2}. Moreover, the lower bound in~\eqref{imp_est} indicates that the (nonnegative) quantity $\frac{C_{\nu}}{2\nu+1}$ does not approach zero when $\nu$ is very large.

\begin{table}[htbp]
\centering
\begin{tabular}{@{} cccccc @{}}
\toprule
$\nu$  &  $\frac{C_{\nu}}{2\nu+1}$  &  $\nu$  &  $\frac{C_{\nu}}{2\nu+1}$  &  $\nu$  &  $\frac{C_{\nu}}{2\nu+1}$  \\
\midrule
1  &  0.444444444444  &  11  &  0.319308836395  &  21  &  0.318593810687  \\
2  &  0.344265186330  &  12  &  0.319154197729  &  22  &  0.318569074526  \\
3  &  0.330166016890  &  13  &  0.319032940515  &  23  &  0.318547437753  \\
4  &  0.325185126688  &  14  &  0.318936089320  &  24  &  0.318528403079  \\
5  &  0.322817680242  &  15  &  0.318857498690  &  25  &  0.318511569147  \\
6  &  0.321499599981  &  16  &  0.318792844850  &  26  &  0.318496609147  \\
7  &  0.320688247098  &  17  &  0.318739013405  &  27  &  0.318483254782  \\
8  &  0.320152568605  &  18  &  0.318693714591  &  28  &  0.318471284130  \\
9  &  0.319780050865  &  19  &  0.318655234248  &  29  &  0.318460512343  \\
10 &  0.319510389628  &  20  &  0.318622268380  &  30  &  0.318450784453  \\
\bottomrule
\end{tabular}
\caption{Numerical evaluation of $\frac{C_{\nu}}{2\nu+1}$ with $1\leq \nu\leq 30$.}
\label{tab:1}
\end{table}

\section{Preliminaries}

\label{sec:pre}

\setcounter{equation}{0}

In this section, we review some useful properties of the Chebyshev polynomials of the first kind.

In view of~\eqref{Cheby2}, if $t\in[-1,1]$, then $|T_{k}(t)|\leq1$ for any $k$. In particular, if $t\in[0,1]$ and $k$ is an odd number, the following estimate (see, e.g.,~\cite[Proposition 6.26]{Vassilevski2008}) holds.

\begin{lemma}
For any $t\in[0,1]$, it holds that
\begin{equation}\label{est_T}
\left|T_{2k+1}(t)\right|\leq\min\big\{1,(2k+1)t\big\}.
\end{equation}
\end{lemma}

The following lemma (see, e.g.,~\cite[Proposition 6.25]{Vassilevski2008}) provides two special expressions for $T_{k}(t)$, which are needed in the proof of our main result.

\begin{lemma}\label{ex_T}
The Chebyshev polynomials defined by~\eqref{Cheby} have the following expressions:
\begin{subequations}
\begin{align}
T_{2k}(t)&=(-1)^{k}+P_{k}(t^{2}),\label{Te}\\
T_{2k-1}(t)&=(-1)^{k-1}(2k-1)t+tQ_{k-1}(t^{2}),\label{To}
\end{align}
\end{subequations}
where $P_{k}(0)=0$, $Q_{k}(0)=0$, and both $P_{k}(x)$ and $Q_{k}(x)$ are polynomials of degree $k$ (in terms of the variable $x$).
\end{lemma}

According to~\eqref{Cheby}, \eqref{Te}, and~\eqref{To}, one can readily deduce that
\begin{displaymath}
P_{0}(t^{2})=0, \quad P_{1}(t^{2})=2t^{2}, \quad Q_{0}(t^{2})=0, \quad \text{and} \quad Q_{1}(t^{2})=4t^{2}.
\end{displaymath}
Furthermore, $P_{k}(t^{2})$ and $Q_{k}(t^{2})$ have the following recursive relations~\cite[Proposition A1]{Brezina2012}.

\begin{lemma}
Let $P_{k}(t^{2})$ and $Q_{k}(t^{2})$ be given as in Lemma~\ref{ex_T}. Then
\begin{subequations}
\begin{align}
P_{k}(t^{2})&=-P_{k-1}(t^{2})+2(-1)^{k-1}(2k-1)t^{2}+2t^{2}Q_{k-1}(t^{2}),\label{ex_P}\\
Q_{k}(t^{2})&=(4t^{2}-1)Q_{k-1}(t^{2})-2P_{k-1}(t^{2})+4(-1)^{k-1}(2k-1)t^{2}.\label{ex_Q}
\end{align}
\end{subequations}
\end{lemma} 

\section{Main result}

\label{sec:main}

\setcounter{equation}{0}

We are now in a position to present the main result of this paper.

\begin{theorem}
For any positive integer $\nu$, the quantity $C_{\nu}$ defined by~\eqref{C_v} satisfies that
\begin{equation}\label{est}
\frac{1}{4}\leq\frac{C_{\nu}}{2\nu+1}\leq\frac{4}{9}.
\end{equation}
\end{theorem}

\begin{proof}
(i) \textit{Lower bound}: The lower bound in~\eqref{est} follows from the relation
\begin{displaymath}
\frac{C_{\nu}}{2\nu+1}=\sup_{t\in(0,1]}\frac{\big|(-1)^{\nu}(2\nu+1)t-T_{2\nu+1}(t)\big|}{[(2\nu+1)t]^{2}}\geq\frac{2-\left|T_{2\nu+1}\left(\frac{2}{2\nu+1}\right)\right|}{4}\geq\frac{1}{4}.
\end{displaymath}
	
(ii) \textit{Upper bound}: Table~\ref{tab:1} shows that the upper bound in~\eqref{est} is valid when $1\leq\nu\leq30$. In what follows, we focus on the case $\nu\geq31$.

Let $\alpha>1$ and $\beta>1$ be two undetermined parameters such that
\begin{equation}\label{para1}
\frac{2\nu+1}{2\big\lfloor\frac{\nu}{\beta}\big\rfloor+1}\leq\alpha\leq\frac{2\nu+1}{\sqrt{2}},
\end{equation}
where $\lfloor m \rfloor$ denotes the largest integer that is not larger than $m$. Partitioning the interval $(0,1]$ into the disjoint parts $\big(0,\frac{\alpha}{2\nu+1}\big]$ and $\big(\frac{\alpha}{2\nu+1},1\big]$, we then have
\begin{equation}\label{decom}
\frac{C_{\nu}}{2\nu+1}=\max\big\{S_{1}(\alpha),S_{2}(\alpha)\big\},
\end{equation}
where
\begin{align*}
S_{1}(\alpha)&:=\sup_{t\in \big(0,\frac{\alpha}{2\nu+1}\big]}\frac{\big|(-1)^{\nu}(2\nu+1)t-T_{2\nu+1}(t)\big|}{[(2\nu+1)t]^{2}},\\ S_{2}(\alpha)&:=\sup_{t\in \big(\frac{\alpha}{2\nu+1},1\big]}\frac{\big|(-1)^{\nu}(2\nu+1)t-T_{2\nu+1}(t)\big|}{[(2\nu+1)t]^{2}}.
\end{align*}

\textit{Upper bound for $S_{1}(\alpha)$}: By~\eqref{To}, we have
\begin{equation}\label{S1}
S_{1}(\alpha)=\sup_{t\in \big(0,\frac{\alpha}{2\nu+1}\big]}\frac{\left|Q_{\nu}(t^{2})\right|}{(2\nu+1)^{2}t}.
\end{equation}
For any nonnegative integer $k$ $(k\leq \nu)$, we define
\begin{displaymath}
R_{k}(x):=2P_{k}(x)-4(-1)^{k}(2k+1)x,
\end{displaymath}
where $P_{k}(\cdot)$ is given by~\eqref{Te}. In view of~\eqref{ex_P} and~\eqref{To}, we have
\begin{align*}
R_{k}(t^{2})&=-2P_{k-1}(t^{2})+4t^{2}Q_{k-1}(t^{2})+(-1)^{k-1}16kt^{2}\\
&=-R_{k-1}(t^{2})+4t^{2}Q_{k-1}(t^{2})+(-1)^{k-1}4(2k-1)t^{2}+(-1)^{k-1}8t^{2}\\
&=-R_{k-1}(t^{2})+4t\big[tQ_{k-1}(t^{2})+(-1)^{k-1}(2k-1)t\big]+(-1)^{k-1}8t^{2}\\
&=-R_{k-1}(t^{2})+4tT_{2k-1}(t)+(-1)^{k-1}8t^{2},
\end{align*}
which yields
\begin{equation}\label{rec_R}
\frac{\left|R_{k}(t^{2})\right|}{t}\leq\frac{\left|R_{k-1}(t^{2})\right|}{t}+4\left|T_{2k-1}(t)\right|+8t\leq\frac{\left|R_{k-1}(t^{2})\right|}{t}+4\min\big\{1, (2k-1)t\big\}+8t.
\end{equation}
In the last inequality, we have used the estimate~\eqref{est_T}. By applying~\eqref{rec_R} recursively, we obtain
\begin{displaymath}
\frac{\left|R_{k}(t^{2})\right|}{t}\leq\frac{\left|R_{0}(t^{2})\right|}{t}+8kt+4\sum_{j=1}^{k}\min\big\{1, (2j-1)t\big\}=4(2k+1)t+4\sum_{j=1}^{k}\min\big\{1, (2j-1)t\big\}.
\end{displaymath}
In addition, from~\eqref{ex_Q}, we have
\begin{displaymath}
Q_{k}(t^{2})=(4t^{2}-1)Q_{k-1}(t^{2})-R_{k-1}(t^{2}).
\end{displaymath}
Note that $|4t^{2}-1|\leq1$ due to $t\in \big(0,\frac{\alpha}{2\nu+1}\big]$ and $\alpha\leq\frac{2\nu+1}{\sqrt{2}}$. We then have
\begin{displaymath}
\frac{\left|Q_{k}(t^{2})\right|}{t}\leq\frac{\left|Q_{k-1}(t^{2})\right|}{t}+4(2k-1)t+4\sum_{j=1}^{k-1}\min\big\{1, (2j-1)t\big\},
\end{displaymath}
which leads to
\begin{equation}\label{estQ}
\frac{\left|Q_{k}(t^{2})\right|}{t}\leq 4k^{2}t+4\sum_{i=2}^{k}\sum_{j=1}^{i-1}\min\big\{1, (2j-1)t\big\}=4k^{2}t+4\sum_{j=1}^{k-1}(k-j)\min\big\{1, (2j-1)t\big\}.
\end{equation}
Using~\eqref{S1} and~\eqref{estQ}, we obtain
\begin{align*}
S_{1}(\alpha)&\leq 4\sup_{t\in \big(0,\frac{\alpha}{2\nu+1}\big]}\frac{\nu^{2}t+\sum\limits_{j=1}^{\nu-1}(\nu-j)\min\big\{1, (2j-1)t\big\}}{(2\nu+1)^{2}}\\
&=\frac{4}{(2\nu+1)^{3}}\left(\nu^{2}\alpha+\sum_{j=1}^{\nu-1}(\nu-j)\min\big\{2\nu+1, (2j-1)\alpha\big\}\right).
\end{align*}

Due to $\alpha\geq\frac{2\nu+1}{2\big\lfloor\frac{\nu}{\beta}\big\rfloor+1}$, it follows that
\begin{displaymath}
\sum_{j=1}^{\nu-1}(\nu-j)\min\big\{2\nu+1, (2j-1)\alpha\big\}\leq\sum_{j=1}^{\big\lfloor\frac{\nu}{\beta}\big\rfloor}(\nu-j)(2j-1)\alpha+(2\nu+1)\sum_{j=\big\lfloor\frac{\nu}{\beta}\big\rfloor+1}^{\nu-1}(\nu-j).
\end{displaymath}
Since $\frac{\nu}{\beta}-1<\big\lfloor\frac{\nu}{\beta}\big\rfloor\leq\frac{\nu}{\beta}$, we have that
\begin{align*}
\sum_{j=1}^{\big\lfloor\frac{\nu}{\beta}\big\rfloor}(\nu-j)(2j-1)\alpha&=\left(\nu\sum_{j=1}^{\big\lfloor\frac{\nu}{\beta}\big\rfloor}(2j-1)-\sum_{j=1}^{\big\lfloor\frac{\nu}{\beta}\big\rfloor}(2j^{2}-j)\right)\alpha\\
&=\left(\nu\Big\lfloor\frac{\nu}{\beta}\Big\rfloor^{2}+\frac{\big\lfloor\frac{\nu}{\beta}\big\rfloor\big(\big\lfloor\frac{\nu}{\beta}\big\rfloor+1\big)}{2}-\frac{\big\lfloor\frac{\nu}{\beta}\big\rfloor\big(\big\lfloor\frac{\nu}{\beta}\big\rfloor+1\big)\big(2\big\lfloor\frac{\nu}{\beta}\big\rfloor+1\big)}{3}\right)\alpha\\
&\leq\frac{\nu\left[(6\beta-4)\nu^{2}+9\beta\nu+\beta^{2}\right]}{6\beta^{3}}\alpha
\end{align*}
and
\begin{displaymath}
\sum_{j=\big\lfloor\frac{\nu}{\beta}\big\rfloor+1}^{\nu-1}(\nu-j)=\frac{\big(\nu-\big\lfloor\frac{\nu}{\beta}\big\rfloor\big)\big(\nu-\big\lfloor\frac{\nu}{\beta}\big\rfloor-1\big)}{2}\leq\frac{\nu\left[(\beta-1)^{2}\nu+\beta^{2}-\beta\right]}{2\beta^{2}}.
\end{displaymath}
Hence,
\begin{equation}
S_{1}(\alpha)\leq\frac{(12\beta-8)\nu^{3}+(12\beta^{3}+18\beta)\nu^{2}+2\beta^{2}\nu}{3\beta^{3}(2\nu+1)^{3}}\alpha+\frac{2(\beta-1)^{2}\nu^{2}+2\beta(\beta-1)\nu}{\beta^{2}(2\nu+1)^{2}}.
\end{equation}

Let
\begin{align*}
u_{1}(\nu)&=\frac{(12\beta-8)\nu^{3}+(12\beta^{3}+18\beta)\nu^{2}+2\beta^{2}\nu}{3\beta^{3}(2\nu+1)^{3}},\\ u_{2}(\nu)&=\frac{2(\beta-1)^{2}\nu^{2}+2\beta(\beta-1)\nu}{\beta^{2}(2\nu+1)^{2}}.
\end{align*}
Direct computations yield
\begin{align*}
u_{1}'(\nu)&=\frac{-24(\beta^{3}+1)\nu^{2}+\beta(24\beta^{2}-8\beta+36)\nu+2\beta^{2}}{3\beta^{3}(2\nu+1)^{4}}, &\lim_{\nu\rightarrow+\infty}u_{1}(\nu)&=\frac{3\beta-2}{6\beta^{3}}, \\
u_{2}'(\nu)&=\frac{-4(\beta-1)\nu+2\beta(\beta-1)}{\beta^{2}(2\nu+1)^{3}}, &\lim_{\nu\rightarrow+\infty}u_{2}(\nu)&=\frac{(\beta-1)^{2}}{2\beta^{2}}.
\end{align*}
Here, $u_{i}'(\nu)$ denotes the derivative of $u_{i}(\nu)$ with respect to $\nu$.
Taking $\beta=\frac{5}{2}$, we have
\begin{align*}
u_{1}(\nu)&=\frac{176\nu^{3}+1860\nu^{2}+100\nu}{375(2\nu+1)^{3}}, &u_{1}'(\nu)&=\frac{-3192\nu^{2}+3320\nu+100}{375(2\nu+1)^{4}}, &\lim_{\nu\rightarrow+\infty}u_{1}(\nu)&=\frac{22}{375},\\
u_{2}(\nu)&=\frac{18\nu^{2}+30\nu}{25(2\nu+1)^{2}}, &u_{2}'(\nu)&=\frac{-24\nu+30}{25(2\nu+1)^{3}}, &\lim_{\nu\rightarrow+\infty}u_{2}(\nu)&=\frac{9}{50}.
\end{align*}
Obviously, $u_{1}'(\nu)<0$ and $u_{2}'(\nu)<0$ if $\nu\geq31$. Thus, if $\nu\geq31$, we have
\begin{equation}\label{part1}
S_{1}(\alpha)\leq u_{1}(31)\alpha+u_{2}(31)=\frac{2344592}{31255875}\alpha+\frac{6076}{33075}.
\end{equation}

\textit{Upper bound for $S_{2}(\alpha)$}: On the other hand, it is easy to see that
\begin{equation}\label{part2}
S_{2}(\alpha)\leq\sup_{t\in \big(\frac{\alpha}{2\nu+1},1\big]}\frac{1}{(2\nu+1)t}+\sup_{t\in \big(\frac{\alpha}{2\nu+1},1\big]}\frac{\left|T_{2\nu+1}(t)\right|}{[(2\nu+1)t]^{2}}\leq\frac{1}{\alpha}+\frac{1}{\alpha^{2}}.
\end{equation}

Combining~\eqref{decom}, \eqref{part1}, and \eqref{part2}, we arrive at
\begin{displaymath}
\frac{C_{\nu}}{2\nu+1}\leq\max\left\{\frac{2344592}{31255875}\alpha+\frac{6076}{33075}, \frac{1}{\alpha}+\frac{1}{\alpha^{2}}\right\}=:U(\alpha).
\end{displaymath}
Straightforward calculations yield that
\begin{displaymath}
\alpha_{\ast}=\mathop{\arg\min}_{\alpha>1}U(\alpha)\approx 3.142703993650 \quad \text{and} \quad U(\alpha_{\ast})\approx 0.419446860530<\frac{4}{9}.
\end{displaymath}
Consequently, for any positive integer $\nu$, it holds that
\begin{displaymath}
\frac{C_{\nu}}{2\nu+1}\leq\frac{4}{9}.
\end{displaymath}
This completes the proof.
\end{proof}

\begin{remark}\rm
When $\beta=\frac{5}{2}$, the constraint~\eqref{para1} reduces to
\begin{equation}\label{para2}
\frac{2\nu+1}{2\big\lfloor\frac{2\nu}{5}\big\rfloor+1}\leq\alpha\leq\frac{2\nu+1}{\sqrt{2}}.
\end{equation}
Because $\big\lfloor\frac{2\nu}{5}\big\rfloor>\frac{2\nu}{5}-1$ and $\nu\geq31$, we have that
\begin{displaymath}
\frac{2\nu+1}{2\big\lfloor\frac{2\nu}{5}\big\rfloor+1}<\frac{10\nu+5}{4\nu-5}\leq\frac{315}{119}\approx 2.647058823529 \quad \text{and} \quad \frac{2\nu+1}{\sqrt{2}}\geq\frac{63}{\sqrt{2}}\approx 44.547727214752.
\end{displaymath}
It is clear that the minimizer $\alpha_{\ast}\approx 3.142703993650$ satisfies~\eqref{para2}.
\end{remark}

\begin{remark}\rm	
Finally, we remark that there is a small mistake in the proofs of~\cite[Proposition A2]{Brezina2012} and~\cite[Proposition 2.2]{Guo2014}. More specifically, the statement $R_{0}(t^{2})=0$ does not hold. In fact, $R_{0}(t^{2})=-4t^{2}$.
\end{remark}

\section*{Acknowledgments}
The authors would like to thank the anonymous referee for his/her valuable comments and suggestions, which greatly improved the original version of this paper. This work was supported by the National Key Research and Development Program of China (Grant No. 2016YFB0201304) and the Key Research Program of Frontier Sciences of CAS. 

\bibliographystyle{abbrv}
\bibliography{references}

\end{document}